\documentclass[cmp]{svjour}

\usepackage[utf8]{inputenc}
\usepackage{graphics, color, amsmath, amsfonts, amssymb, amscd}
\usepackage{constants}

\newconstantfamily{cs}{
symbol=c,
}

\spnewtheorem{itRemark}{Remark}{\bf}{\it}

\newcommand{\llangle}{\left\langle}
\newcommand{\rrangle}{\right\rangle}
\newcommand{\benum}{\begin{enumerate}}
\newcommand{\eenum}{\end{enumerate}}
\newcommand{\bitm}{\begin{itemize}}
\newcommand{\eitm}{\end{itemize}}

\newcommand{\lrarrow}{\leftrightarrow}

\newcommand{\couronne}{\Delta_R}
\newcommand{\Cluster}{\mathbf{C}}
\newcommand{\VCluster}{\mathbf{c}}

\newcommand{\icomp}{\mathrm{i}}
\newcommand{\dd}{\mathrm{d}}
\newcommand{\dtheta}{\dd^\Lambda\theta}
\newcommand{\norm}[1]{\|#1\|}
\newcommand{\norminf}[1]{\norm{#1}}
\newcommand{\normsup}[1]{\|#1\|_{\scriptscriptstyle\infty}}
\newcommand{\normI}[1]{\|#1\|_{\scriptscriptstyle 1}}
\newcommand{\normII}[1]{\|#1\|_{\scriptscriptstyle 2}}
\newcommand{\setof}[2]{\{#1\,:\,#2\}}

\newcommand{\Bsetof}[2]{\Bigl\{#1\,:\,#2\Bigr\}}
\newcommand{\littleo}[1]{o(#1)}
\newcommand{\bbZ}{\mathbb{Z}}
\newcommand{\bbR}{\mathbb{R}}
\newcommand{\bbN}{\mathbb{N}}
\newcommand{\bbS}{\mathbb{S}}
\newcommand{\bbP}{\mathbb{P}}
\newcommand{\bbE}{\mathbb{E}}
\newcommand{\bbQ}{\mathbb{Q}}
\newcommand{\calA}{\mathcal{A}}
\newcommand{\calB}{\mathcal{B}}
\newcommand{\calE}{\mathcal{E}}
\newcommand{\calF}{\mathcal{F}}
\newcommand{\calG}{\mathcal{G}}
\newcommand{\calR}{\mathcal{R}}
\newcommand{\calS}{\mathcal{S}}
\newcommand{\IF}[1]{\boldsymbol{1}_{\{#1\}}}
\renewcommand{\emptyset}{\varnothing}

\providecommand{\openbox}{\leavevmode
  \hbox to.77778em{%
  \hfil\vrule
  \vbox to.675em{\hrule width.6em\vfil\hrule}%
  \vrule\hfil}}
\makeatletter
\DeclareRobustCommand{\qed}{%
  \leavevmode\unskip\penalty9999 \hbox{}\nobreak\hfill
  \quad\hbox{\openbox}%
}
\makeatother

\title{Upper bound on the decay of correlations\\in a general 
class of $O(N)$-symmetric models}
\titlerunning{Decay of correlations in $O(N)$ models}
\author{Maxime Gagnebin \and Yvan Velenik}
\institute{Section de mathématiques, Université de Genève, 
\email{Maxime.Gagnebin@unige.ch}, \email{Yvan.Velenik@unige.ch}}
\date{}

\begin{document}

\maketitle

\begin{abstract}
We consider a general class of two-dimensional spin systems, with 
continuous but not necessarily smooth, possibly long-range, $O(N)$-symmetric 
interactions, for which we establish algebraically decaying upper bounds on 
spin-spin correlations under all infinite-volume Gibbs measures.

As a by-product, we also obtain estimates on the effective resistance of a
(possibly long-range) resistor network in which randomly selected edges are 
shorted.
\end{abstract}

\keywords{Spin systems -- continuous symmetry -- decay of correlations -- 
McBryan-Spencer bound -- Mermin-Wagner theorem}

\tableofcontents

\section{Introduction, statement of results}

\subsection{Settings and earlier results}

We consider the following class of lattice spin systems. To each site 
$i\in\bbZ^2$, we associate a spin $S_i\in\bbR^N$ such that $\normII{S_i}=1$. 
Given $\Lambda\Subset\bbZ^2$, we define the Hamiltonian
\begin{equation}\label{eq:Hamiltonian}
H_\Lambda(S) = -\sum_{(u,v)\in\calE_\Lambda} 
J_{u,v} V(S_u,S_v),
\end{equation}
where $\calE_\Lambda=\setof{(x,y)}{\{x,y\}\cap\Lambda\neq\emptyset}$ is the set 
of all (unoriented) edges intersecting $\Lambda$.

The coupling constants are assumed to satisfy $J_{u,v}=J_{u-v}=J_{v-u}\geq 0$ 
and $\sum_{x\in\bbZ^2} J_x < \infty$; we shall actually assume, without loss 
of generality, that $\sum_{x\in\bbZ^2} J_x = 1$.

The interaction $V$ is assumed to be invariant under simultaneous rotation 
of its two arguments; in other words, it is assumed that $V(S_u,S_v)$ depends
only on the scalar product $S_u\cdot S_v$.

The corresponding finite-volume Gibbs measure in $\Lambda$, with boundary 
condition $\bar S$ and at inverse temperature $\beta$ is then defined by
\[
\mu_{\Lambda;\beta}^{\bar S} (\dd^\Lambda S) =
\begin{cases}
\frac 1{Z_{\Lambda;\beta}^{\bar S}}\, e^{-\beta H_\Lambda(S)} \, 
\dd^\Lambda S
& \text{if $S_y=\bar S_y$ for all $y\not\in\Lambda$,}\\
0	& \text{otherwise,}
\end{cases}
\]
where we used the notation $\dd^\Lambda S=\prod_{x\in\Lambda} \dd S_x$, with
$\dd S_x$ denoting the Haar measure on $\bbS^{N-1}$. The expectation with respect to this
measure will be denoted~$\langle\cdot\rangle_{\Lambda, \beta}^{\bar S}$ or $\langle\cdot\rangle_\mu$.
The conditions we will impose on $V$ in Section~\ref{sec:AssumptionsAndResults} 
 will guarantee
the existence of the Gibbs measure above (that is, that 
$\infty>Z_{\Lambda;\beta}^{\bar S}>0$).

\medskip
It is well-known that in this setting, under some mild conditions, all 
infinite-volume Gibbs measures associated to such a system are necessarily 
rotation-invariant. Since the classical result of Mermin and 
Wagner~\cite{MerminWagner}, numerous works have been devoted to strengthening 
the claims and weakening the hypotheses. 
In the present context, the strongest statement to date is proven 
in~\cite{ISV}. In the latter paper, rotation invariance of the infinite-volume 
Gibbs measures was established under the following assumptions:
\begin{itemize}
\item the random walk on $\bbZ^2$ with transition probabilities from 
$u$ to $v$ given by $J_{u,v}$ is recurrent;
\item the interaction $V$ is continuous (actually, a weaker integral 
condition is also given there).
\end{itemize}
The recurrence assumption is known to be optimal in general, as there are 
versions of the two-dimensional $O(N)$ model for which rotation invariance is 
spontaneously broken at sufficiently low temperatures, 
whenever the random walk is transient (see~\cite[Theorem~20.15]{Georgii} and
references therein). The absence of any smoothness 
assumption on $V$ was the main contribution of~\cite{ISV}, such assumptions 
having played a crucial role in earlier approaches (see, for example, 
\cite{%
DobrushinShlosman,Pfister,FrohlichPfister,KleinLandauShucker,BonatoPerezKlein,%
Ito}).

\bigskip
A particular consequence of the rotation-invariance of the infinite-volume 
Gibbs measures is the fact that spin-spin correlations $\langle S_0 \cdot 
S_x\rangle$ vanish as $\norminf{x}\to\infty$. A natural problem is then to 
quantify the speed of decay of these quantities.

The first class of systems that have been studied had finite-range (usually 
nearest-neighbor) interactions. An upper bound on the decay of the form
${\langle S_0 \cdot S_x\rangle} \leq (\log\norminf{x})^{-c}$, $c>0$, was first
derived by Fisher and Jasnow~\cite{FisherJasnow} for the $O(N)$ models
(which correspond to $V(S_u,S_v)=S_u\cdot S_v$), $N\geq 2$. Their result was 
then 
extended by McBryan and Spencer~\cite{McBrianSpencer}, who obtained an 
algebraically 
decaying upper bound of the form $\langle S_0 \cdot S_x\rangle \leq 
\norminf{x}^{-c/\beta}$, which is best possible in general. Indeed, Fröhlich
and 
Spencer have proved an algebraically decaying \emph{lower} bound of that type 
for the two-dimensional $XY$ model ($O(2)$ model) at low 
temperatures~\cite{FrohlichSpencer}. Building
on~\cite{DobrushinShlosman}, Shlosman managed to
obtain upper bounds of the same type for a much larger class of
interactions~\cite{Shlosman}, under some smoothness assumption on
$V$. A similar, but less general, result was later obtained by 
Naddaf~\cite{Naddaf}, using an adaptation of the McBryan-Spencer approach. More 
recently, Ioffe, Shlosman and Velenik showed how to dispense with the 
smoothness assumption, extending Shlosman's result to very general
interactions $V$~\cite{ISV}.

The first results for models with infinite-range interactions provided an upper 
bound \textit{à la} Fisher-Jasnow~\cite{BonatoPerezKlein,Ito} for models with
$J_x$ such that the corresponding random walk is recurrent. An algebraically
decaying upper bound was obtained by Shlosman~\cite{Shlosman} for a general
class of models (with a smoothness assumption on $V$) in the case of
exponentially decaying coupling constants. Algebraically decaying upper bounds
were obtained for $O(N)$ models by Messager, Miracle-Sole and 
Ruiz~\cite{Marseillais}, when the coupling constants satisfy $J_x\leq 
C\norminf{x}^{-\alpha}$ with $\alpha>4$.

\subsection{Assumptions and results}\label{sec:AssumptionsAndResults}

Let us now turn to the results contained in the present paper. Since, as is 
explained in~\cite{McBrianSpencer}, the case $N\geq 3$ can be reduced to the 
case $N=2$, we restrict our attention to the latter. In that case, it is 
convenient to parametrize the spins by their angle, that is, to associate to 
each vertex $u\in\bbZ^2$ the angle $\theta_u\in[0,2\pi)$ such that 
$S_u=(\cos\theta_u,\sin\theta_u)$. One can then rewrite the interaction as
\[
\beta V(S_u,S_v) = f(\theta_u-\theta_v),
\]
which we shall do from now on. (Of course, addition is  mod $2\pi$ and $f$ is
even.) 
Our analysis will rely on two assumptions~\footnote{In all this paper, we use 
the notation $a\lesssim b$ when there exists a constant $C$, depending maybe on 
the value of $\alpha$ below, but not on any other parameters, such that $a\leq C 
b$. When $a\lesssim b$ and $b\lesssim a$, we write $a\sim b$.}:
\begin{itemize}
\item[\textbf{A1.}] There exist $\alpha>4$  and $J\geq 0$ such that $J_x\leq  
J \normI{x}^{-\alpha}$, for all $x\in\bbZ^d$.
\item[\textbf{A2.}] The function $f$ is continuous.
\end{itemize}

The main result of this paper is the following claim, which substantially 
increases the range of models for which algebraically decreasing upper bounds 
on spin-spin correlations can be established.
\begin{theorem}\label{thm:Main}
Under Assumptions A1 and A2, there exist $C$ and $c$ (depending on the 
interaction $f$ and the coupling constants $(J_x)_{x\in\bbZ^d}$) such that, 
for any infinite-volume Gibbs measure $\mu$ associated to the 
Hamiltonian~\eqref{eq:Hamiltonian},
\[
\left|\langle \cos(\theta_0-\theta_x)\rangle_\mu\right|
\leq
c\norminf{x}^{-C},\qquad\forall x\in\bbZ^d.
\]
\end{theorem}
\begin{itRemark}\label{rem:CsurBeta}
The above result does not specify the dependence of the constant $C$ on the 
inverse temperature $\beta$. To obtain such an information, the method 
of proof we use requires some assumptions on the smoothness of $f$. For 
example, as explained in Remark~\ref{rem:DepOnBeta}, if $f$ is assumed to be 
$s$-Hölder, for some $s>3$, then we obtain that $\left|\langle 
\cos(\theta_0-\theta_x)\rangle_\mu\right| \leq c\norminf{x}^{-C/\beta}$ for 
large enough values of $\beta$.
\end{itRemark}
\begin{itRemark}
We have assumed above that the coupling constants $(J_x)_{x\in\bbZ^d}$ were 
nonnegative, but as can easily be checked from the proof, this assumption can 
be removed, as long as $\sum_x |J_x|<\infty$.
\end{itRemark}
\begin{itRemark}
Note that, when $J_x\sim\norm{x}^{-\alpha}$ with $\alpha<4$, the result is 
actually not true in general, as long-range order has been established in that 
case in some 
models~\cite{KunzPfister}.
\end{itRemark}
The proof of Theorem~\ref{thm:Main} is given in 
Section~\ref{sec:ProofOfMainThm}. It is based on a suitable combination of the 
McBryan-Spencer approach, in the form derived in~\cite{Marseillais}, with the
expansion technique developed in~\cite{ISV} in order to deal with not 
necessarily smooth interactions.

\subsection{A result on a random resistor network}

The proof of Theorem~\ref{thm:Main} is closely linked to the properties 
of a random resistor network, and can be used to extract some information on 
the latter. We refer to~\cite{DoyleSnell} for a gentle introduction to resistor 
networks.

We now interpret $1/J_{x,y}$ as the resistance of a wire between the 
vertices $x$ and $y$. The quantity
\[
\calE_{\text{eff}}
=
\min\Bsetof{\frac12\sum_{u,v\in\bbZ^2} 
J_{u,v}\bigl(g(u)-g(v)\bigr)^2}{g(0)=1,g(x)=0} 
\]
represents the energy dissipated by the network when a voltage of $1$
volt is imposed between $0$ and $x$. It is related in a simple way to 
the effective resistance $\calR(0,x)$ of the resistor network:
\[
\calR(0,x) = \bigl( \calE_{\text{eff}} \bigr)^{-1}.
\]
It is classical in random walk theory (see, for example, 
\cite[Theorem~4.4.6]{LawlerLimic}) that under the Assumption~\textbf{A1} on
the coupling constants $(J_x)_{x\in\bbZ^d}$, the effective resistance satisfies
\[
\calR(0,x) \sim\log\norminf{x}.
\]
Note that an assumption of type~\textbf{A1} is necessary: there are examples 
for 
which $\calR(0,x)$ grows more slowly than $\log\norm{x}$. It is the case if the 
transition probabilities $J_{x,y}$ are proportional to $1/\norm{x-y}^4$, as can 
be shown by a direct computation~\footnote{\label{footnote_loglog}This is 
ultimately due to the change of behavior of the characteristic function of the 
increments of the random walk associated to the coupling constant 
$(J_x)_{x\in\bbZ^d}$: assume that $J_x=C\norm{x}^s$, 
then~\cite[(3.11)]{BiskupRP}, as $k\to 0$,
\[
1-\sum_x J_x e^{\icomp k\cdot x} \sim
\begin{cases}
\norm{k}^2			& \text{if $s>4$,}\\
\norm{k}^2\log(1/\norm{k})	& \text{if $s=4$,}\\
\norm{k}^{s-2}			& \text{if $s<4$.}
\end{cases}
\]
This implies, in particular, that the random walk is transient when 
$s<4$, and that the potential kernel $a(x)$ satisfies $a(x)\sim\log\norm{x}$ 
when $s>4$, but $a(x)\sim\log\log\norm{x}$ when $s=4$.
}.
An adaptation of the proof of Theorem~\ref{thm:Main} yields an extension of 
this result to a random resistor model, in which the resistance 
$R_{x,y}$ between two vertices $x$ and $y$ is taken to be
\[
R_{x,y}=
\begin{cases}
1/J_{x,y}	&	\text{\rm  with probability }1-\epsilon J_{x,y},\\ 
0		&	\text{\rm with probability }\epsilon J_{x,y},
\end{cases}
\]
the choice being done independently for each pair $(x,y)$.
In other words, randomly selected resistors are ``shorted'' (which amounts to 
identifying endpoints). This, of course, can only lower the effective 
resistance, compared to the above deterministic case. The following result 
shows that, when $\epsilon$ is small enough, this decrease typically does not 
change the qualitative behavior of $\mathcal R(0,x)$.
\begin{theorem}\label{thm:Effective}
Under Condition~\textbf{A1}, we have that, for all $\epsilon$ small enough, 
there exist $C,\kappa$ and $\tilde c>0$ such that, uniformly in $x\in\bbZ^2$,
\[
\bbP\bigl(\mathcal R(0,x)\geq\tilde c \log \norminf{x} \bigr) \geq 
1-\frac{C}{\norminf{x}^{\kappa}}.
\]
\end{theorem}
\begin{itRemark}
Of course, the upper bound
\[
\bbP\bigl(\mathcal R(0,x)\geq\tilde c \log \norminf{x} \bigr)
\leq
1-\epsilon J_{x},
\]
always holds, since $\calR(0,x)=0$ whenever $R_{0,x}=0$.
\end{itRemark}

\subsection{Open problems}

\begin{itemize}
\item We have only obtained an explicit dependence of the decay exponent on 
$\beta$ when $f$ is sufficiently smooth, as described in 
Remark~\ref{rem:CsurBeta}. It would be interesting to determine whether such a 
behavior can be established for any continuous interaction.
\item Our results above are restricted to $\alpha>4$. It may seem more natural 
to assume only recurrence of the random walk with transition probabilities $J_x$
(note that, as already mentioned, if $J_x\sim \norm{x}^{-\alpha}$ with 
$\alpha<4$, then the corresponding random walk is transient). Indeed, 
recurrence is the optimal condition for the validity of Mermin-Wagner type 
theorems. However, it seems unlikely (for the reason outlined in 
Footnote~\ref{footnote_loglog}) that the correlations admit an 
algebraically decaying upper bound in the whole recurrence regime. 
Some upper bounds with slower decay have been obtained under 
weaker conditions than~\textbf{A1}  (in particular, when $\alpha=4$)
in~\cite{Marseillais} in the case of the $O(N)$ model. It would be interesting 
to clarify these issues.
\item As mentioned above, negative coupling constants can be accommodated in 
our approach, but only in a very rough way. In general, situations in which the 
sign of the coupling constants is not constant should allow an extension of 
this type of results to interactions with slower decay; see~\cite{Bruno} for a 
discussion of the $XY$ and Heisenberg models with oscillatory interactions 
of the form $J_x\sim\cos(a\norm{x}+b)/\norm{x}^\alpha$. 
Alternatively, one can consider disordered $O(N)$-models in which the coupling 
constants are i.i.d.~random variables with zero mean (satisfying suitable 
moment conditions); algebraic decay of correlations has been established in 
this setting in~\cite{Picco} and~\cite{vanEnter}. It would be interesting to 
determine whether our approach also applies in such situations.
\item Of course, by far the most important open problem in this area is the 
proof of the conjecture stating that the spin-spin correlations in short-range 
$O(N)$ models with $N\geq 3$ decay exponentially fast at any 
temperature~\cite{Simon15}.
\end{itemize}

\begin{acknowledgement}
The authors are grateful to the two referees for a very careful 
reading, several corrections and useful suggestions. This research was 
partially supported by the Swiss National Science Foundation.
\end{acknowledgement}

\section{Proof of Theorem~\ref{thm:Main}}
\label{sec:ProofOfMainThm}
We consider the model in a large box $\Lambda_M=\{-M,\ldots,M\}^2$, with fixed, 
but arbitrary, boundary condition $\bar\theta$. Our goal in this section is to 
prove that there exist $C$ and $c$ such that
\begin{equation}\label{eq:TargetEstimate}
\bigl|\langle\cos(\theta_0-\theta_x)\rangle_{\Lambda_M}^{\bar\theta}\bigr|
\leq
c\norminf{x}^{-C},
\end{equation}
uniformly in $M,\bar\theta$ and $x\in\Lambda_M$.
The main claim easily follows from~\eqref{eq:TargetEstimate} and the DLR 
equations. Indeed, for an arbitrary infinite-volume Gibbs measure $\mu$,
\[
\bigl|\langle\cos(\theta_0-\theta_x)\rangle_{\mu}\bigr|
=
\bigl|\bigl\langle\langle\cos(\theta_0-\theta_x)\rangle_{\Lambda_M}^{\cdot}
\bigr\rangle_{\mu}\bigr|
\leq
\bigl\langle\bigl|\langle\cos(\theta_0-\theta_x)\rangle_{\Lambda_M}^{\cdot}
\bigr|\bigr\rangle_{\mu}
\leq
c\norminf{x}^{-C}.
\]

\bigskip
The remainder of this section is devoted to the proof 
of~\eqref{eq:TargetEstimate}. As we shall always work in the fixed big box 
$\Lambda_M$, with the fixed boundary condition $\bar\theta$ there won't be any 
ambiguity if we use the
following lighter notations: $\langle\cdot\rangle 
\equiv \langle\cdot\rangle_{\Lambda_M}^{\bar\theta}$, $H\equiv 
H_{\Lambda_M}$, $Z=Z_{\Lambda_M}^{\bar\theta}$ and $\calE\equiv 
\calE_{\Lambda_M}$. Also, we shall write $\norm{x}\equiv\normsup{x}$.

\subsection{Warm-up: $f$ given by a trigonometric polynomial}
\label{ssec:simple}

We first consider an easier case, in which the interaction $f$ can be written 
as a trigonometric polynomial,
\[
f(x)=\sum_{k=1}^K c_k\cos(kx).
\]
Although this case can be treated by a straightforward adaptation of the 
arguments in~\cite{Marseillais}, we include the complete argument here, as 
we shall need it when considering the general case in 
Subsection~\ref{sec:GeneralCase}.

\subsubsection{McBryan and Spencer's complex rotation}

In~\cite{McBrianSpencer}, McBryan and Spencer used a complex rotation of the 
spin variables in order to reduce the analysis of the correlations of the 
nearest-neighbor $O(N)$ model to a variational problem. The same can be done 
here. First, observe that
\begin{align*}
\left|\langle \cos(\theta_0-\theta_x)\rangle\right|
&\leq\left|\langle\cos(\theta_0-\theta_x)
+ \icomp\sin(\theta_0-\theta_x)\rangle\right|
=
\bigl|\frac1 Z \int e^{\icomp(\theta_0-\theta_x)- H(\theta)} 
\dtheta\bigr|.
\end{align*}
Now, thanks to the periodicity and analyticity of the function 
$\icomp(\theta_0-\theta_x)- H(\theta)$, applying the (inhomogeneous) complex 
rotation $\theta_z\to\theta_z+\icomp a_z$ leaves the integral unchanged as long 
as $a_z=0$ for all $z\not\in\Lambda$. 
Therefore,
\begin{multline*}
\Bigl|\frac1 Z \int e^{\icomp(\theta_0-\theta_x)- 
H(\theta)}d\theta\Bigr|
=
\Bigl|\frac1 Z \int e^{\icomp(\theta_0 + \icomp a_0 - \theta_x - 
\icomp a_x) - H(\theta+\icomp a)} \dtheta\Bigr|\\
\leq
e^{a_x-a_0} \frac1 Z \int \exp\Bigl\{\sum_{(u,v)\in\calE} J_{u,v} 
\sum_{k=1}^K c_k\cos(k\nabla_{u,v}\theta) \cosh(k\nabla_{u,v} a)\Bigr\}\dtheta,
\end{multline*}
where we used the identity $\cos(x+i a)=\cos(x)\cosh(a)-i \sin(x)\sinh(a)$ and 
the notation $\nabla_{u,v}a=a_v-a_u$.
We now reconstruct the original Gibbs measure by adding and subtracting 
$ H(\theta)$:
\begin{align}
e^{a_x-a_0}& \frac1 Z \int \exp\Bigl\{\sum_{(u,v)\in\calE} J_{u,v} 
\sum_{k=1}^K c_k\cos(k\nabla_{u,v}\theta) \cosh(k\nabla_{u,v} a)\Bigr\} 
\dtheta\notag\\
&=
e^{a_x-a_0} \frac1 Z \int 
\exp\Bigl\{\sum_{(u,v)\in\calE}J_{u,v} \sum_{k=1}^K 
c_k\cos(k\nabla_{u,v}\theta)\notag\\
&\hspace*{5cm}
\times \bigl(\cosh(k\nabla_{u,v} 
a)-1\bigr) - H(\theta)\Bigr\} \dtheta\notag\\
&=
e^{a_x-a_0} \bigl\langle \exp\Bigl\{\sum_{(u,v)\in\calE} 
J_{u,v}\sum_{k=1}^K 
c_k\cos(k\nabla_{u,v}\theta) \bigl(\cosh(k\nabla_{u,v} 
a)-1\bigr)\Bigr\}\bigr\rangle\notag\\
&\leq
\exp\Bigl\{ a_x-a_0 +\sum_{(u,v)\in\calE} J_{u,v}\sum_{k=1}^K
|c_k| \bigl(\cosh(k\nabla_{u,v}a)-1\bigr) \Bigr\}
.
\label{eq:PlusMinusH}
\end{align}

\subsubsection{Choice of $a$}

The problem is now reduced to obtaining a good upper bound on
\begin{equation}
\calF(a) = \exp\Bigl\{ a_x-a_0+\sum_{(u,v)\in\calE} J_{u,v} 
\sum_{k=1}^K |c_k| \bigl(\cosh(k\nabla_{u,v} a)-1\bigr)\Bigr\}.
\label{eq_fonctionelle}
\end{equation}
As in~\cite{Marseillais}, we define a particular function $a$ and check 
that it gives the desired decay of correlation. Let us first define the 
function $\bar a:\bbN \to \bbR$ by $\bar a(0) = 0$ and
\begin{equation}\label{eq:abar}
\bar a_i = -\delta \sum_{j=1}^i \frac 1j,\qquad i\geq 1,
\end{equation}
where the parameter $\delta$ will be chosen (small) later.

It will be convenient to decompose $\bbZ^2$ into layers: 
$L_i=\setof{z\in\bbZ^2}{\norminf{z} = i}$. Note that $|L_i|=8i$ for all 
$i\geq1$.

Let $R=\norminf{x}$. Our choice of the function $a$ will be radially symmetric:
\[
a_z =
\begin{cases}
\bar a_{\norminf{z}} - \bar a_R	& \text{if $\norminf{z}\leq R$,}\\
0			& \text{otherwise.}
\end{cases}
\]
Let us now return to the derivation of an upper bound on $\calF(a)$. Note 
that $a_u=a_v=\bar a_i$ whenever $u,v$ belong to the same layer $L_i$, $i\leq 
R$; in particular, in such a case, $\cosh(k\nabla_{u,v}a)-1 = 0$. Therefore,
\begin{align*}
\calF(a)
&\leq
\exp\Bigl\{ \bar a_R + \sum_{i=0}^R\sum_{u\in L_i} \sum_{j\geq 1} 
\sum_{v\in L_{i+j}} J_{u,v} \sum_{k=1}^K |c_k|
\bigl( \cosh(k(\nabla_{u,v} a)) - 1 \bigr) \Bigr\}\notag\\
&=
\exp\Bigl\{ \bar a_R + 
\sum_{i=0}^R\sum_{j\geq1} \sum_{k=1}^K |c_k| \bigl( \cosh(k( \bar a_i - \bar 
a_{i+j})) - 
1 \bigr) \sum_{u\in L_i}\sum_{v\in L_{i+j}} J_{u,v} \Bigr\}\notag\\
&\leq
\exp\Bigl\{ \bar a_R + 
8J2^\alpha \sum_{i=0}^R |L_i| \sum_{j\geq 1} j^{1-\alpha} \sum_{k=1}^K |c_k|
\bigl( \cosh(k( \bar a_i - \bar a_{i+j})) - 1 \bigr) \Bigr\},
\end{align*}
since
\[
\sum_{u\in L_i} \sum_{v\in L_{i+j}} J_{u,v}
\leq
|L_i| 8 \sum_{\ell\geq 0} J_{j+\ell}
\leq 
8 |L_i| \sum_{\ell\geq 0} \frac J{(j+\ell)^\alpha}
\leq
8 J 2^{\alpha}\frac {|L_i|}{j^{\alpha-1}}.
\]
To estimate the sum over $j$, we treat separately the cases $j>i$ and 
$j\leq i$. Let
\[
C_K=\sum_{k=1}^K |c_k|  \quad\text{ and }\quad D_K=\sum_{k=1}^K |c_k| k^2 .
\]
We start with the case $j>i$. Since $\cosh(z)-1\leq e^{z}$ for all 
$z\geq 0$,
\begin{align*}
\sum_{j>i} j^{1-\alpha} \sum_{k=1}^K |c_k|
\bigl\{ \cosh(k(\bar a_i - \bar a_{i+j})) - 1 \bigr\}
&\leq
C_K \sum_{j>i} j^{1-\alpha} \bigl\{ \cosh(K(\bar a_i - \bar a_{i+j})) - 1 
\bigr\}\\
&\leq
C_K \sum_{j>i} j^{1-\alpha} \bigl\{
\cosh\Bigl(K\delta\log\bigl((i+j)/i\bigr)\Bigr)-1 
\bigr\}\\ 
&\leq
C_K \sum_{j>i} j^{1-\alpha} \exp\bigl\{ K\delta\log\bigl((i+j)/i\bigr)\bigr\}\\
&\leq
C_K \sum_{j>i} j^{1-\alpha} \exp\bigl\{ K\delta\log (2j/i) \bigr\}\\
&\leq
C_K \sum_{j>i} j^{1-\alpha} (\frac{2j}i)^{K\delta}\\
&\leq 
2 C_K i^{2-\alpha} .
\end{align*}
Note that we need $\delta$ small enough to ensure that $\alpha-1-K\delta>1$;
this can be done, for example, by choosing $\delta\leq 2/K$ (remember 
that $\alpha>4$). It will actually be convenient to assume the stronger 
condition $K\delta\leq 1$, which we already used above to make the constants in 
the last line a bit simpler.

To treat the case $j\leq i$ we use the fact that $\cosh(t)-1\leq \tfrac23 t^2$ 
when $|t|<1$:
\begin{align*}
\sum_{j=1}^i j^{1-\alpha} \sum_{k=1}^K |c_k| \bigl( 
\cosh(k\delta\sum_{\ell=i+1}^{i+j} 1/\ell\, )-1 \bigr)
&\leq
\sum_{j=1}^i j^{1-\alpha} \sum_{k=1}^K |c_k| \bigl(\cosh(k\delta j/i)-1 
\bigr)\\ 
&\leq
\sum_{j=1}^i j^{1-\alpha} \sum_{k=1}^K |c_k| \tfrac23 \bigl(k\delta j/i 
\bigr)^2\\
&\leq
\tfrac23 D_K \delta^2 i^{-2} \sum_{j\geq 1} j^{3-\alpha}\\
&=
\tfrac 23 D_K\delta^2 i^{-2} \zeta(\alpha-3),
\end{align*}
where $\zeta(\cdot)$ denotes the Riemann zeta function (notice that 
$\zeta(\alpha-3)<\infty$ since $\alpha>4$).
Bringing all the pieces together, we see that 
\begin{align}
\calF(a)
&\leq
\exp\Bigl\{ \bar a_R +
64J2^\alpha \sum_{i=1}^R i 
\bigl( 2 C_K i^{2-\alpha} + \tfrac23 D_K\delta^2 \zeta(\alpha-3) i^{-2} 
\bigr) + 8J2^\alpha C_K \zeta(\alpha-2)\Bigr\} \notag\\
&\leq
\exp \Bigl\{ \bigl(-\delta + \delta^2 \tfrac23  D_K
64J2^\alpha\zeta(\alpha-3)\bigr)\log R +  72 C_K J 2^\alpha 
\zeta(\alpha-3)\Bigr\},
\label{eq_bornefin}
\end{align}
provided that we choose $\delta$ such that the factor multiplying 
$\log R$ be negative (and $K\delta\leq 1$). This can always be done. 
With such a choice, we conclude that there exist
$c,C<\infty$, uniform in $x,M,\bar\theta$, such that
\[
\left|\langle \cos(\theta_0-\theta_x)\rangle\right|
\leq
c\norminf{x}^{-C}.
\]
\begin{itRemark}
Note that, if we write explicitly the dependence on $\beta$ in the above 
expressions, then $D_k$ is actually $\beta D_K$. This yields the classical 
bound
\[
|\langle \cos(\theta_0-\theta_x)\rangle|
\leq c(\beta)\norminf{x}^{-C/\beta}.
\]
\end{itRemark}

\subsection{General case}
\label{sec:GeneralCase}
We now turn our attention to the general case of functions satisfying assumption
\textbf{A2.}

\subsubsection{Measure decomposition}\label{ssec:MeasureDecomposition}

In order to treat general interactions, we proceed as in~\cite{ISV}. Namely, we 
first fix $\epsilon>0$ (which will be assumed to be small enough later on). 
Trigonometric polynomials being dense (w.r.t. the sup-norm) in the set of 
continuous functions on $\bbS^1$, it is possible to find a number 
$K=K(\epsilon)$ such that
\begin{alignat*}{2}
f(x)
&=
\sum_{k=0}^K c_k\cos(kx)	&&+\bar\epsilon(x)\\
&=\tilde f(x)			&&+\bar\epsilon(x),
\end{alignat*}
with $\bar\epsilon$ satisfying the conditions
\[
\normsup{\bar\epsilon}\leq\epsilon \quad\text{ and }\quad\bar\epsilon\geq 0.
\]
Note that the constant $c_0$ doesn't affect the Gibbs measure and can thus be 
assumed to be equal to $0$, which we will do from now on.

Following~\cite{ISV}, we then make a high-temperature expansion of the 
singular part:
\begin{align*}
Z
&=
\int e^{- H(\theta)}\dtheta
=
\int  \exp\Bigl\{ \sum_{(u,v)\in\calE} J_{u,v}
\bigl(\bar\epsilon(\nabla_{u,v}\theta) +
\tilde f(\nabla_{u,v}\theta)\bigr)\Bigr\}
\dtheta\\
&=
\int \exp\Bigl\{ \sum_{(u,v)\in\calE} J_{u,v} \tilde f(\nabla_{u,v}\theta) 
\Bigr\}
\prod_{(u,v)\in\calE}
\Bigl( e^{ J_{u,v} \bar\epsilon(\nabla_{u,v}\theta)} - 1 + 1 \Bigr)
\dtheta\\
&=
\sum_{A\in\calA} \int
\exp\Bigl\{ \sum_{(u,v)\in\calE} J_{u,v} \tilde f(\nabla_{u,v}\theta) 
\Bigr\} 
\prod_{(u,v)\in A} \Bigl(e^{ J_{u,v} \bar\epsilon(\nabla_{u,v}\theta)} - 1 
\Bigr) \dtheta\\
&=\sum_{A\in\calA} Z_A,
\end{align*}
where we have introduced the set $\calA=\{A\subset\calE\}$.
This allows us to decompose the original measure as the convex combination
\[
\langle g(\theta) \rangle
=
\sum_{A\in\calA} \pi(A) \, \langle g(\theta) \rangle_A,
\]
where $\pi(A) = Z_A/Z$ and
\[
\langle g(\theta) \rangle_A
=
\frac1{Z_A} \int g(\theta)
\exp\Bigl\{
\sum_{(u,v)\in\calE} J_{u,v} \tilde f(\nabla_{u,v}\theta)
\Bigr\} 
\prod_{(u,v)\in A}
\Bigl( e^{ J_{u,v}\bar\epsilon(\nabla_{u,v}\theta)} - 1 \Bigr)
\dtheta.
\]
If we take a look at the quantity of interest,
$|\langle\cos(\theta_x-\theta_0)\rangle|$,
we can split the sum over $A\in\calA$ into two: a set $\calG$ 
of ``good'' configurations, and a set $\calB$ of ``bad'' configurations, thus 
leading to an upper bound 
\begin{align*}
|\langle\cos(\theta_x-\theta_0)\rangle|
&=
\Bigl|
\sum_{A\in\calG} \langle\cos(\theta_x-\theta_0)\rangle_A \pi(A)
+
\sum_{A\in\calB} \langle\cos(\theta_x-\theta_0)\rangle_A \pi(A)
\Bigr|\\
&\leq
\sum_{A\in\calG}\pi(A)\, |\llangle\cos(\theta_x-\theta_0)\rrangle_A 
|+\sum_{A\in 
\calB}\pi(A).
\end{align*} 
We will choose $\calG$ and $\calB$ in such a way that the 
quantities ${|\langle\cos(\theta_x-\theta_0)\rangle_A |}$ can be estimated 
appropriately, while keeping the probability $\pi(\calB)$ sufficiently small.

\subsubsection{Working with $\llangle \cdot\rrangle_A$}

The above decomposition is very helpful because it fixes the set $A$. We 
shall see how the complex rotation method can also be used in this case.
\begin{align*}
|\langle\cos(\theta_0-\theta_x)\rangle_A|
&\leq
\Bigl|\frac1{Z_A} \int
\exp\Bigl\{ \icomp (\theta_0-\theta_x)
+ \sum_{(u,v)\in\calE} J_{u,v} \tilde f(\nabla_{u,v}\theta)\Bigr\}\\
&\hspace*{4cm}
\times\prod_{(u,v)\in A} \Bigl( e^{
J_{u,v} \bar\epsilon(\nabla_{u,v}\theta)}-1\Bigr)
\dtheta\Bigr|.
\end{align*}

Of course, the fact that one has absolutely no information 
on the smoothness of the function $\bar\epsilon$ has to be addressed now. Since 
this function is not analytic in general, one cannot directly apply 
the translation $\theta\to\theta+\icomp a$. The key observation is that the 
interaction can be factorized into an interaction on each cluster of $A$ and 
another interaction between different clusters of $A$. 

Let $\Cluster$ be one of the clusters of $A$, and let us denote its 
vertices by $\{\VCluster_1,\dots,\VCluster_n\}$. Assume first that 
$0,x\notin\Cluster$ and $\Cluster\subset\Lambda$. We can then factorize the 
above integral as follows:
\[
\Bigl| \frac1{Z_A}\int\dd^{\Lambda\backslash\Cluster}\theta \,
(\,\dots)\int\dd^\Cluster\theta \; F(\theta,\Cluster) 
\prod_{\substack{u\in\Cluster\\ v\notin\Cluster}}e^{J_{u,v}\tilde 
f(\nabla_{u,v}\theta)}\Bigr|
\]
where $(\,\dots)$ represents the terms depending only on the variables
$(\theta_i,i\notin\Cluster)$, and
\[
F(\theta,\Cluster)
=
\prod_{\substack{u,v\in\Cluster\\(u,v)\in 
A}}\bigl(e^{J_{u,v}\bar\epsilon(\nabla_{u,v}\theta)}-1\bigr)\prod_{
u,v\in\Cluster}e^{J_{u,v}\tilde f(\nabla_{u,v}\theta)} .
\] 
For $i=2,\dots,n$, let us define 
$\eta_i=\theta_{\VCluster_i}-\theta_{\VCluster_1}$. Since the function $F$ 
depends only on the values of the gradients of $\theta$ inside the cluster 
$\Cluster$ and is therefore a function of $\eta=(\eta_i,i=2,\ldots,n)$, 
changing variables from $(\theta_{\VCluster_i}, i=1,\ldots,n)$ to 
$(\theta_{\VCluster_1},\eta_2,\ldots,\eta_n)$ yields
\[
\int\dd^\Cluster\theta\;F(\theta,\Cluster)\prod_{\substack{
u\in\Cluster\\v\notin\Cluster}}e^{J_{u,v}\tilde f(\nabla_{u,v}\theta)}
=
\int \prod_{i=2}^n\dd\eta_i\;F(\eta,\Cluster)\int\dd\theta_{\Cluster_1}
\prod_{\substack{u\in\Cluster\\v\notin\Cluster}}e^{J_{u,v}\tilde 
f(\nabla_{u,v}\theta)}.
\]
The function in the last integral is now analytic and periodic in 
$\VCluster_1$. We are thus entitled to make the complex shift
$\theta_{\VCluster_1}\mapsto\theta_{\VCluster_1}+\icomp a_{\VCluster_1}$. 
In terms of the original variables, this shift corresponds to
$\theta_{\VCluster_i}\mapsto\theta_{\VCluster_1}+\icomp 
a_{\VCluster_1}+\eta_i$, which implies that the whole cluster $\Cluster$ 
is shifted by the same amount $a_{\VCluster_1}$. This procedure can be 
reproduced on each cluster of $A$ (including isolated vertices) to obtain a 
global shift $\theta\mapsto\theta+\icomp a$ with the constraint that $a_u=a_v$ 
whenever $u$ and $v$ belong to the same cluster of $A$.

In the preceding discussion, we have made the hypothesis that 
$0,x\notin\Cluster$. The case where either $0,x$ or both belong to $\Cluster$ 
is 
treated in exactly the same way. The only change is that we apply the 
complex shift to $e^{\icomp(\theta_0-\theta_x)}\prod e^{J_{u,v}\tilde 
f(\nabla_{u,v}\theta)}$ which is also an analytic function.

Finally, the assumption that $\Cluster\subset\Lambda$ cannot be discarded. If a 
point of the boundary condition is in $\Cluster$ then this point ``fixes'' the 
whole cluster, which cannot be shifted. We thus add the constraint that 
$a\equiv 0$ on the connected component of the exterior of $\Lambda$.

\medskip
In order to emphasize the above constraints, we shall henceforth write $a^A$ 
instead of $a$. We thus have, proceeding as in~\eqref{eq:PlusMinusH},
\begin{align*}
|\langle\cos(\theta_0-&\theta_x)\rangle_A|\\
&\leq
e^{a^A_x-a^A_0}\frac1{Z_A}
\int \exp\Bigl\{ 
\sum_{(u,v)\in\calE} J_{u,v} \sum_{k=1}^K c_k\cos(k\nabla_{u,v}\theta) 
\cosh(k\nabla_{u,v}a^A) \Bigr\}\\
&\hspace*{5.5cm}
\times\prod_{(u,v)\in A} \Bigl( e^{
J_{u,v}\bar\epsilon(\nabla_{u,v}\theta)} - 1\Bigr)
\dtheta\\
&\leq
e^{a^A_x-a^A_0}
\exp\Bigl\{ \sum_{(u,v)\in\calE} J_{u,v}
\sum_{k=1}^K |c_k| \bigl\{\cosh(k\nabla_{u,v}a^A)-1\bigr\}\Bigr\}
=
\mathcal F(a^A) .
\end{align*}
We thus get the exact same result as in Subsection~\ref{ssec:simple} (but, of 
course, with additional constraints on admissible $a^A$).
%
\subsubsection{Good and Bad sets of edges}

In order to proceed, we must now provide a suitable definition of the sets 
$\calG$ and $\calB$ and prove that they have the required properties. To this 
end, we need some terminology.
\begin{definition}
Given $A\in\calA$, we say that $u$ and $v$ are connected if there 
exists a sequence $x_0,x_1,\dots,x_n$ such that $x_0=u$, $x_n=v$ and 
$(x_i,x_{i+1})\in A$ for all $i=0,\ldots,n-1$; we denote this by $u\lrarrow 
v$. We need the following three quantities: for any $u\in\bbZ^2$, set
\begin{itemize}
\item[] $m_A(u)=\max\setof{\norminf{v}}{v\lrarrow u}$, the norm of the 
furthest point connected to $u$; 
\item[] $n_A(u) = |\setof{v}{v\lrarrow u}|$, the number of points connected to 
$u$;
\item[] $r_A(u)=m_A(u)-\norminf{u}$, the ``outwards radius" of the cluster of $u$.
\end{itemize}
\end{definition}

\begin{definition} \label{def_goodconfig}
Let $\couronne=\Lambda_R\setminus\Lambda_{R^{1/2}}$. We say that a 
configuration of open edges $A$ is $\Cl[cs]{cst_nArA2}$-\emph{good} if:
\benum
\item For all $u\in\Lambda_{R^{1/2}}$, we have that $u\nleftrightarrow 
\Lambda_{2R^{1/2}}^c$, where
$\Lambda_{2R^{1/2}}^c=\bbZ^2\setminus\Lambda_{2R^{1/2}}$. 
\label{pas_de_connection_hors_N13}
\item For all $u\in\couronne$, we have that $m_A(u)\leq 2\norminf{u}$ 
(equivalently, $r_A(u)\leq \norminf{u}$). 
\label{pas_de_longues_aretes}
\item $\displaystyle\sum_{u\in\couronne}\frac{r_A(u)^2}{\norminf{u}^2}\leq 
\Cr{cst_nArA2}\log R$. \label{somme_k2Nik}
\eenum
We then define $\calG=\setof{A\in\calA}{A\text{ is 
$\Cr{cst_nArA2}$-good}}$ and $\calB=\calA\setminus\calG$.
\end{definition}

\subsection{Estimate on the good set $\calG$}\label{ssec:EstimateGood}

Let us now see how the approach from~\cite{Marseillais} can be used to obtain 
the same estimate we had in equation~\eqref{eq_bornefin}, when
$A\in\calG$. As we have seen above, the complex rotation argument leads to the 
following bound:
\begin{equation}
|\langle\cos(\theta_0-\theta_x)\rangle_A|
\leq
\exp\Bigl\{a^A_x - a^A_0
+ \sum_{k=1}^K |c_k| \sum_{(u,v)\in\calE} J_{u,v} 
\bigl\{\cosh\bigl(k(a^A_u-a^A_v)\bigr)-1\bigr\}
\Bigr\}.
\label{borne1}
\end{equation}
We now have to make a choice for the function $a^A:\bbZ^2\to\mathbb R$, 
compatible with our two requirements that $a^A_u=a^A_v$ whenever $u\lrarrow 
v$, and $a^A\equiv 0$ outside $\Lambda$. To make that choice, we first modify 
the function used in Subsection~\ref{ssec:simple}, making $a^A$ grow only for 
points sufficiently  far from the origin. Namely, we first define, for 
$i\in\mathbb N$, (remember~\eqref{eq:abar})
\begin{equation*}
\tilde a_i=
\begin{cases}
0				& \text{if $i\leq 2R^{1/2}$,}\\
\bar a_i-\bar a_{2R^{1/2}}	& \text{if $i\in\{2R^{1/2},\ldots,R\}$,}\\
\bar a_R-\bar a_{2R^{1/2}}	& \text{if $i\geq R$.}
\end{cases}
\end{equation*}
The actual rotation $a^A_u$ is then defined similarly to what we did in 
Subsection~\ref{ssec:simple}, using $\tilde a$ instead of $\bar a$, but taking 
its value on the furthest point $v$ to which $u$ is connected, thus ensuring 
that it remains constant on each cluster:
\begin{equation}\label{def_newa}
a^A_u= \tilde a_{m_A(u)} - \tilde a_{R}.
\end{equation}
Now, since $f(x+y+z)\leq f(3x)+f(3y)+f(3z)$ for any nonnegative increasing 
function $f$, we can write
\begin{multline*}
\cosh\bigl(k(\tilde a_{m_A(u)}-\tilde a_{m_A(v)})\bigr) - 1\\
\leq
\bigl\{\cosh\bigl(3k(\tilde a_{m_A(u)}-\tilde a_{\norminf{u}})\bigr) - 1\bigr\}
+\bigl\{\cosh\bigl(3k(\tilde a_{\norminf{u}}-\tilde a_{\norminf{v}})\bigr) - 
1\bigr\}\\
+\bigl\{\cosh\bigl(3k(\tilde a_{\norminf{v}}-\tilde a_{m_A(v)})\bigr) - 
1\bigr\}.
\end{multline*}
After multiplying them by $J_{u,v}$ and summing over $u,v\in\bbZ^2$, the 
contributions from the first and third terms above will be the same. Since 
$\sum_{v\in\bbZ^2}J_{u,v}=1$, we get
\begin{multline}
\sum_{u,v\in\bbZ^2}J_{u,v} \bigl\{ 
\cosh\bigl(k(a^A_u-a^A_v)\bigr)-1\bigr\}\\
\leq
2\sum_{u\in\bbZ^2} \bigl\{\cosh\bigl(3k(\tilde a_{m_A(u)}-\tilde 
a_{\norminf{u}})\bigr) - 1\bigr\}\\
+
\sum_{u,v\in\bbZ^2}J_{u,v} \bigl\{ \cosh\bigl(3k(\tilde 
a_{\norminf{u}}-\tilde a_{\norminf{v}})\bigr)-1\bigr\}.
\label{eq:twoterms}
\end{multline}
The contribution from the second term in the right-hand side 
of~\eqref{eq:twoterms} is bounded as in Subsection~\ref{ssec:simple}:
\[
\sum_{u,v\in\bbZ^2} \sum_{k=1}^K |c_k| \bigl\{
\cosh\bigl(3k(\tilde a_{\norminf{u}}-\tilde a_{\norminf{v}})\bigr) - 1 
\bigr\}
\leq
\Cl[cs]{cst_add} C_K + \Cl[cs]{cst_mult}D_K\delta^2\log R.
\]
To estimate the first term in the right-hand side of~\eqref{eq:twoterms}, we 
rely on the fact that $A$ is assumed to be good: choosing $\delta$ small enough 
so that $K\delta<1/3$,
\begin{align*}
\sum_{u\in\bbZ^2} \sum_{k=1}^K |c_k| \bigl\{
\cosh\bigl(3k(\tilde a_{m_A(u)}-\tilde a_{\norminf{u}})\bigr) - 1
\bigr\}
&\leq
\sum_{u\in\couronne} \sum_{k=1}^K |c_k| \bigl\{
\cosh\bigl(3k(\bar a_{m_A(u)} - \bar a_{\norminf{u}} ) \bigr) - 1
\bigr\}\\
&=
\sum_{u\in\couronne} \sum_{k=1}^K |c_k| \bigl\{
\cosh\bigl(3k\delta\sum_{\ell=\norminf{u}+1}^{m_A(u)}\frac1\ell \bigr) - 1
\bigr\}\\
&\leq
\sum_{u\in\couronne} \sum_{k=1}^K |c_k| 
\bigl(3k\delta\frac{r_A(u)}{\norminf{u}}\bigr)^2\\
&\leq
9 D_K \delta^2\Cr{cst_nArA2}\log R .
\end{align*}
The last piece of~\eqref{borne1} left to estimate is $a^A_x-a^A_0$,
with $x\in L_R$,
\[
a^A_x - a^A_0
=
-\delta\Bigl( \sum_{i=1}^R \frac 1 i - \sum_{i=1}^{2R^{1/2}} \frac 1 i \Bigr)
\leq
-\tfrac18\delta\log R .
\]
(The factor $\tfrac18$ in the last expression is just here to get rid of the 
constant additional terms, a better bound is 
$\delta+\delta\log(2)-\delta\tfrac12\log(R)$.) We are now ready to 
bring the pieces together,
\[
|\langle\cos(\theta_0-\theta_x)\rangle_A|
\leq
\exp\bigl\{
-\tfrac18 \delta \log R
+(\Cr{cst_add} C_K
+\Cr{cst_mult}D_K\delta^2\log R
+9 D_K\delta^2\Cr{cst_nArA2}\log R)
\bigr\} .
\]
Again taking $\delta$ sufficiently small for the constant multiplying $\log R$ 
to be negative.
\[
|\langle\cos(\theta_0-\theta_x)\rangle_A|
\leq
cR^{-C} .
\]
\begin{itRemark}
\label{rem:DepOnBeta}
We now explain how additional information on the smoothness of $f$ can be used 
in order to obtain an explicit dependence of the constant $C$ above on the 
inverse temperature $\beta$.

The main place such an information is useful is when we approximate $f$ by 
the trigonometric polynomial $\tilde f$ and the error term $\bar\epsilon$. 
Indeed, in order for the percolation arguments of Appendix~\ref{app:A} to work, 
we need $\epsilon$ to be small enough (which ensures a sufficiently subcritical 
percolation model). If the $\beta$ dependence is spelled out explicitly, this 
means that we need $\epsilon\lesssim 1/\beta$. To ensure this, we need to let 
the number $K$ of terms appearing in $\tilde f$ grow. The question is: how 
fast?

Let us assume that $f$ is $s$-Hölder for some $s>3$. In that case, we can draw 
two conclusions~\cite{Jackson,Grafakos}: If $\hat f(x) = \sum_{k=1}^\infty \hat 
c_k \cos(kx)$ is the Fourier series associated to $f$, then
\begin{itemize}
\item $\sup_x |f(x) - \sum_{k=1}^K \hat c_k \cos(kx)| \lesssim K^{-s}\log K$.
\item $|\hat c_k| \lesssim k^{-s}$, and thus 
$\sum_{k=1}^\infty |\hat c_k| k^2 = D_\infty < \infty$.
\end{itemize}
We can thus set $\tilde f=\sum_{k=1}^K \hat c_k \cos(kx)$ with $K\lesssim 
\sqrt\beta$ (any $\littleo{\beta}$ would work). We then choose 
$\delta=\delta_1/\beta$, for $\delta_1$ small enough. In that case:
\begin{itemize}
\item the condition $K\delta\leq1/3$ is automatically satisfied for large $\beta$.
\item Rewriting explicitly the dependance of $D_K$ in $\beta$, $D_K\lesssim\beta D_\infty$
 yields the classical upper bound 
$|\langle\cos(\theta_0-\theta_x)\rangle_A|
\leq c(\beta)R^{-C/\beta}$.
\end{itemize}
\end{itRemark}

\subsection {The bad set $\calB$ has small probability}\label{ssec:BadSet}

Of course, not all configurations are good. The rest of this section is devoted 
to showing that bad configurations have an appropriately small probability. 
The following observation is very 
useful. 
\begin{proposition}\label{prop:DomByPerco}
On the set $\calA$, with the natural partial ordering, 
the measure $\pi$ (defined in Section~\ref{ssec:MeasureDecomposition}) is
dominated by the independent Bernoulli percolation 
process $\bbP$ in which an edge $(x,y)$ is opened with probability $2\epsilon
 J_{x,y}$.
\end{proposition} 
\begin{proof}
To establish this domination, we will show that the probability for an edge
$(x,y)$ to be open is bounded by $2\epsilon J_{x,y}$ 
uniformly in the states of all the other edges. \\
Let $D\in\calA$ such that $(x,y)\notin D$. It suffices~\cite{LSS} to show that
\[
\frac{\pi(D\cup (x,y))}{\pi(D)+\pi(D\cup (x,y))}
\leq
2\epsilon J_{x,y},
\]
which will clearly follow if we prove that $\pi(D\cup (x,y)) \leq 2\epsilon 
J_{x,y} \pi(D)$. This is straightforward using the definition of $\pi$:
\begin{align*}
\pi(D\cup (x,y))
&=\frac1Z \int\dtheta \exp\{\sum_{x,y\in\bbZ^2}J_{x,y}f(\nabla_{x,y}\theta)\}\times\\
&\hspace*{2.5cm}\times\prod_{(u,v)\in
D}\bigl(e^{J_{u,v}\bar\epsilon(\theta_u-\theta_v)}-1\bigr)\Big(e^{
J_{x,y}\bar\epsilon(\theta_x-\theta_y)}-1\Big)\\
&\leq 2\epsilon J_{x,y}\pi(D),
\end{align*}
where we have used that $0\leq\bar\epsilon(x)\leq\epsilon$, for all $x$. This 
concludes the proof.
\qed
\end{proof}

The three properties characterizing the set $\calB$ are bounds on increasing 
functions on $\calA$. Proposition~\ref{prop:DomByPerco} thus 
implies that $\pi(\calB) \leq \bbP(\calB)$, which means that we can henceforth 
work with the measure $\bbP$ instead of $\pi$.

We shall consider the three conditions characterizing $\calB$ one at a time.

\paragraph{Condition~\ref{pas_de_connection_hors_N13}}
For $u\in\Lambda_{R^{1/2}}$, Proposition~\ref{prop_estimatebridge} (see 
Appendix~\ref{app:A}) yields the following estimate:
\[
\bbP(u\lrarrow \Lambda_{2R^{1/2}}^c)
\leq
\sum_{k\geq R^{1/2}} \bbP(r_A(u)=k)
\leq
\sum_{k\geq R^{1/2}}\frac{\Cr{surconnect}}{k^{\alpha-1}}
\lesssim
R^{-(\alpha-2)/2}.
\]
Hence, 
\[
\bbP(\Lambda_{R^{1/2}}\lrarrow\Lambda_{2R^{1/2}}^c)
\leq
\sum_{u\in\Lambda_{R^{1/2}}} \bbP(u\lrarrow\Lambda_{2R^{1/2}}^c)
\lesssim
R\cdot R^{-(\alpha-2)/2}
=
R^{-(\alpha-4)/2},
\]
which decreases algebraically with $R$, since $\alpha>4$.

\paragraph{Condition~\ref{pas_de_longues_aretes}}

The proof is similar. For $u\in L_i$, we have that 
\[
\bbP(m_A(u)\geq 2\norminf{u})
=
\sum_{k\geq i} \bbP(m_A(u)=i+k)
\leq
\sum_{k\geq i} \frac{\Cr{surconnect}}{k^{\alpha-1}}
\lesssim
i^{2-\alpha}.
\]
Thus,
\[
\bbP(\exists u\not\in\Lambda_{R^{1/2}}\,:\, m_A(u)\geq 2\norminf{u})
\lesssim
\sum_{i\geq R^{1/2}} \sum_{u\in L_i} i^{2-\alpha}
\lesssim
R^{-(\alpha-4)/2},
\]
which is also algebraically decreasing.

\paragraph{Condition~\ref{somme_k2Nik}}
In order to control $\sum_{u\in\couronne}\frac{r_A(u)^2}{\norminf{u}^2}$, it is 
convenient to introduce a new family $(N(u),R(u))_{u\in\couronne}$ of i.i.d.  
random variables with the same distribution as $(n_A(0),r_A(0))$; their joint 
law will be denoted by $\bbQ$. It is proven in Proposition~\ref{sommenArA2} 
of Appendix~\ref{App:B} that the following holds:
\begin{align}
\bbP\Bigl(\sum_{u\in\couronne}\frac{r_A(u)^2}{\norminf{u}^2}
> \Cr{cst_nArA2} \log R\Bigr)
\leq
\bbQ\Bigl(\sum_{u\in\couronne}\frac{N(u)R(u)^2}{\norminf{u}^2}
>\Cr{cst_nArA2} \log R
\Bigr).
\end{align}
A first indication that the latter probability is small is given by the 
expectation of the sum. Thus,
\begin{align*}
\bbE\Bigl[\sum_{u\in\couronne}\frac{N(u)R(u)^2}{\norminf{u}^2}\Bigr]
&=
\sum_{u\in\couronne}\frac 1 {\norminf{u}^2}\bbE\Bigl[N(u)R(u)^2\Bigr]
\leq 
\sum_{i=1}^R\frac1{i^2}\sum_{u\in 
L_i}\bbE\Bigl[n_A(0)r_A(0)^2\Bigr]\\
&\leq
\sum_{i=1}^R\frac 8{i} \sum_{k\geq 0}\bbP(n_A(0)r_A(0)^2>k).
\end{align*}
The latter probability can be bounded using 
Proposition~\ref{prop_estimatebridge} and Lemma~\ref{lem_boundnA}:
\begin{align}
\bbP(n_A(0)r_A(0)^2>k)
&\leq
\bbP(n_A(0)>\log k) + \bbP\bigl(r_A(0)>\sqrt{k/\log k}\bigr)\notag\\
&\lesssim
\epsilon^{(1/2)\log k}
+ \sum_{t>\sqrt{k/\log k}} \Cr{surconnect}\,t^{1-\alpha}\notag\\
&\lesssim
\epsilon^{(1/2)\log k}
+ (k/\log k)^{(2-\alpha)/2}\notag\\
&\lesssim
(k/\log k)^{(2-\alpha)/2},
\label{eq_nArA2}
\end{align}
provided that $\epsilon$ be chosen small enough. The latter expression is 
summable, since $\alpha>4$, and we obtain
\[
\bbE\Bigl[\sum_{u\in\couronne}\frac{N(u)R(u)^2}{\norminf{u}^2}\Bigr]
\leq
\Cl[cs]{cst_espNR2}\log R .
\]
Let us now define the event $\calS=\{N(u)R(u)^2\leq\norminf{u}^2, 
\forall u\in\couronne\}$. The probability of $\mathcal S^c$ is 
easily bounded using the estimate in equation~\eqref{eq_nArA2}:
\begin{align*}
\bbQ(S^c)
&\leq
\sum_{u\in\couronne}\bbQ(N(u)R(u)^2>\norminf{u}^2)
\lesssim
\sum_{u\in\couronne} (\norminf{u}/\sqrt{\log\norminf{u}})^{2-\alpha}\\
&\lesssim
\sum_{i=R^{1/2}}^R i^{3-\alpha}(\log i)^{(\alpha-2)/2}
\lesssim
R^{(4-\alpha)/2}(\log R)^{(\alpha-2)/2}.
\end{align*}
Note that this probability is algebraically decreasing in $R$. Then, choosing 
$\Cr{cst_nArA2}=2\Cr{cst_espNR2}$, we have
\begin{align*}
\bbQ\Bigl(
\sum_{u\in\couronne}\frac{N(u)R(u)^2}{\norminf{u}^2}>\Cr{cst_nArA2}\log R \Bigr)
&\leq
\bbQ\Bigl(\Bigl\{
\sum_{u\in\couronne}\frac{N(u)R(u)^2}{\norminf{u}^2}>\Cr{cst_nArA2}\log 
R\Bigr\}\cap\calS \Bigr)\\
&+
\bbQ\Bigl(
\sum_{u\in\couronne}\frac{N(u)R(u)^2}{\norminf{u}^2}>\Cr{cst_nArA2}\log R 
\Bigm| \calS^c \Bigr)\, \bbQ(\calS^c)\\
&\leq
\bbQ\Bigl(
\sum_{u\in\couronne} \bigl(\frac{N(u)R(u)^2}{\norminf{u}^2} 
\wedge 1 \bigr) > \Cr{cst_nArA2}\log R \Bigr)\\
&+
\bbQ(\calS^c).
\end{align*}
We have already bounded the second term in the right-hand 
side. To bound the first term, we use Lemma~\ref{lem_sommeva} (note 
that truncating the summands can only decrease the expectation):
\[
\bbQ\Bigl(
\sum_{u\in\couronne}\bigl(\frac{N(u)R(u)^2}{\norminf{u}^2} 
\wedge 1 \bigr) > \Cr{cst_nArA2}\log R \Bigr)
\leq
R^{-\Cl[cs]{exposant}}.
\]
This concludes the bound for Condition~\ref{somme_k2Nik}.

\bigskip
The desired upper bound on the probability of $\calB$ follows.


\section{Proof of Theorem~\ref{thm:Effective}}
\label{sec:EffectiveResistance}
Let us denote by $A$ the (random) set of all edges with $0$ resistance.
In view of our target estimate and our bound on the probability of $\calB$, we 
can assume that $A\in\calG$. In this case, a variant of the computations done 
in Subsection~\ref{ssec:simple} yields
\begin{align*}
\sum_{u,v\in\bbZ^2} J_{u,v} \bigl(a^A_u-a^A_v\bigr)^2
&=
\sum_{u,v\in\bbZ^2} J_{u,v} \bigl(\tilde a_{m_A(u)}-\tilde a_{m_A(v)}\bigr)^2\\
&
\leq
6\sum_{u, v\in\couronne} J_{u,v}
\bigl(\tilde a_{m_A(u)}-\tilde a_{\norminf{u}}\bigr)^2 
+ 3\sum_{u, v\in\bbZ^2} J_{u,v}
\bigl(\tilde a_{\norminf{u}}-\tilde a_{\norminf{v}}\bigr)^2\\
&\leq 
6\sum_{u\in\couronne}\bigl(\delta\frac{r_A(u)}{\norminf{u}}\bigr)^2
+6\sum_{i=1}^R \sum_{j\geq 1} \sum_{u\in L_i} \sum_{v\in L_{i+j}}
J_{u,v} (\delta j/i)^2\\
&\leq
6\Cr{cst_nArA2}\delta^2\log R
+6\sum_{i=1}^R \sum_{j\geq 1} 64J\frac{i}{j^{\alpha-1}} (\delta j/i)^2\\
&\leq
6\Cr{cst_nArA2}\delta^2\log R 
+384\delta^2J\sum_{i=1}^R \frac 1 i \sum_{j\geq 1} j^{3-\alpha}\\
&\lesssim
\delta^2\log R .
\end{align*}
Now, in order to have $a(x)=0$ and $a(0)=1$, $\delta$ must satisfy
$\delta^2\sim\log(R^{1/2})^2=\log(R)^2/4$. This leads to the bound
\[
\calE_{\text{eff}} \lesssim \frac1{\log R} ,
\]
and the result follows.

\appendix
\section{Appendix}

\subsection{Percolation estimates}
\label{app:A}

In this subsection, we collect a number of elementary results on long-range 
percolation, that are needed in the proof of Theorem~\ref{thm:Main}.
Below, $A$ always denotes a configuration of the Bernoulli bond percolation 
process $\bbP$ on $\bbZ^2$, in which an edge $(x,y)$ is open with probability 
$\epsilon J_{x,y}$. The quantities of interest are
\begin{itemize}
\item[] $n_A(x)=|\setof{y}{y\lrarrow x}|$, the cardinality of the cluster of 
$x$;
\item[] $m_A(x)=\max\setof{\norminf{y}}{y\lrarrow x}$, the norm of the furthest 
vertex connected to $x$;
\item[] $r_A(x)=m_A(x)-\norminf{x}$, the ``radius'' of the cluster of $x$.
\end{itemize}
\begin{lemma}\label{lem_boundnA}
For any $\epsilon<\tfrac12$,
\[
\bbP(n_A(0)>k)\lesssim\epsilon^{k/2} .
\]
\end{lemma}
\begin{proof}
For any finite connected graph $G=(V,E)$, it is possible to find a 
path $\gamma$ of length $2|E|$ crossing each edge of $G$ exactly twice, 
starting from any vertex of $G$. This implies that
\begin{align*}
\bbP(n_A(0)>k)
&\leq
\sum_{n\geq k/2}\sum_{\substack{G=(V,E)\\\,V\ni 0, |E|=n}} \bbP(e\text{ is 
open}, \forall e\in E)
=
\sum_{n\geq k/2}\sum_{\substack{G=(V,E)\\\,V\ni 0, |E|=n}} \prod_{e\in E} 
\epsilon J_e\\
&\leq
\sum_{n\geq k/2} \sum_{\substack{\gamma:\\\gamma(0)=0, |\gamma|= 2n}} 
\prod_{e\in\gamma} \sqrt{\epsilon J_e}
\leq
\sum_{n\geq k/2} \epsilon^n \Bigl( \sum_{x\in\bbZ^2} \sqrt{J_x} \Bigr)^{2n},
\end{align*}
and the conclusion follows since $\sum_{x\in\bbZ^2}
\sqrt{J_x}\lesssim\sum_{i\geq1}i^{(\alpha-2)/2}=\zeta((\alpha-2)/2)$. \qed

\end{proof}

\begin{proposition}\label{prop_estimatebridge}
For any $\epsilon$ small enough, there exists $\Cl[cs]{surconnect}$ such that, 
for $x\in\bbZ^2$ and $k>\norminf{x}$, we have the following bound
\[
\bbP(m_A(x)=k) \leq \frac{\Cr{surconnect}}{(k-\norminf{x})^{\alpha-1}} .
\]
\end{proposition}
\begin{proof}
Let $D_m(x,k)$ be the event that there exists an edge-self-avoiding path of 
length $m$ from $x$ to $L_k$, staying inside $\Lambda_{k-1}$ and using only edges 
from $A$:
\begin{multline*}
D_m(x,k) = \bigl\{
\exists\, x_0,\dots,x_m\,: (x_{i-1},x_i)\in A\;\forall i=1,\ldots,m;\\
(x_{i-1},x_i) \neq (x_{j-1},x_j) \text{ if $i\neq j$};
x_0=x;
x_1,\ldots,x_{m-1}\in\Lambda_{k-1};
\norminf{x_m}=k
\bigr\}.
\end{multline*}
Obviously, $\{m_A(x)=k\} \subset \bigcup_{m\geq 1} D_m(x,k)$. We will prove 
below that
$\bbP(D_m(x,k))
\leq 
{(\epsilon 16\Cl[cs]{surconnectm})^m}/{(k-\norminf{x})^{\alpha-1}}$,
which will conclude our proof, since
\[
\bbP(m_A(x)=k)
\leq
\sum_{m\geq 1} \bbP(D_m(x,k))
\leq 
\sum_{m\geq 1} 
\frac{(\epsilon 16\Cr{surconnectm})^m}{(k-\norminf{x})^{\alpha-1}}
\leq
\frac{\Cr{surconnect}}{(k-\norminf{x})^{\alpha-1}}.
\]
We are left with the proof of the bound on $\bbP(D_m(x,k))$, which is done by 
induction. For $m=1$,
\begin{align*}
\sum_{x_1\in L_k} \bbP((x,x_1)\in A)
&\leq
8 \sum_{j\geq k-\norminf{x}} \frac{\epsilon J}{j^\alpha}\\
& \leq \frac {\epsilon 16J}{(k-\|x\|)^{\alpha-1}}.
\end{align*}
Assuming now that the claim is true for any $m\leq M-1$, let us prove it 
for $m=M$:
\begin{align*}
\bbP(D_m(x,k))
&\leq
\sum_{y\in\Lambda_{k-1}} \bbP((x,y)\in A) \bbP(D_{m-1}(y,k))\\
&=
\sum_{\ell=0}^{k-1} \sum_{y\in L_\ell} \bbP((x,y)\in A) \bbP(D_{m-1}(y,k))\\
&\leq
\sum_{\ell=0}^{k-1} 
\frac{(\epsilon 16\Cr{surconnectm})^{m-1}}{(k-\ell)^{\alpha-1}}
\sum_{y\in L_\ell} \bbP((x,y)\in  A)\\
&\leq
\bigl(\sum_{\ell=0}^{\norminf{x}-1} + 
\sum_{\ell=\norminf{x}+1}^{k-1}\bigr)
\frac{(\epsilon 16\Cr{surconnectm})^{m-1}}{(k-\ell)^{\alpha-1}} 
\frac{\epsilon16J}{|\ell-\norminf{x}|^{\alpha-1}}\\
&\hspace*{3cm}
+\frac{(\epsilon 16\Cr{surconnectm})^{m-1}}{(k-\norminf{x})^{\alpha-1}} 
\sum_{y\in L_{\norminf{x}}}\bbP((x,y)\in A)\\
&\leq(\epsilon 16\Cr{surconnectm})^{m-1}\epsilon 16J
\Bigl\{ 
\sum_{\ell=0}^{\norminf{x}-1} 
\frac{1}{(k-\norminf{x})^{\alpha-1}} \frac{1}{(\norminf{x}-\ell)^{\alpha-1}}\\
&\hspace*{3cm}
+\sum_{\ell=\norminf{x}+1}^{k-1} \frac1{(k-\ell)^{\alpha-1}} 
\frac1{(\ell-\norminf{x})^{\alpha-1}}
\\
&\hspace{4cm}
+\frac{\zeta(\alpha-1)}{(k-\norminf{x})^{\alpha-1}} 
\Bigr\}\\
&\leq
(\epsilon 16\Cr{surconnectm})^{m-1} \epsilon 16J
\Bigl\{
\frac{2\zeta(\alpha-1)}{(k-\norminf{x})^{\alpha-1}}\\
&\hspace*{4.2cm}
+\sum_{\ell=1}^{k-\norminf{x}-1} \frac1{(k-\norminf{x}-\ell)^{\alpha-1}}
\frac1{\ell^{\alpha-1}}
\Bigr\}\\
&\leq 
(\epsilon 16\Cr{surconnectm})^{m-1} \epsilon 16J
\zeta(\alpha-1)(2+2^\alpha) \frac1{(k-\norminf{x})^{\alpha-1}},
\end{align*}
where we have used Lemma~\ref{lem_sum(k-l)l} of Appendix~\ref{App:B} in the 
second to last inequality. This ends the proof with $\Cr{surconnectm}=J 
\zeta(\alpha-1)(2+2^\alpha)$. 
\qed
\end{proof}

\begin{proposition}\label{sommenArA2}
The random variable $\displaystyle\sum_{x\in\couronne}\frac{r_A(x)^2}{\|x\|^2}$ 
is stochastically dominated by
$\displaystyle\sum_{x\in\couronne}\frac{N(x)R(x)^2}{\|x\|^2}$, where the random 
variables $(N(x),R(x))_{x\in\Lambda_R}$ are independent and have the same 
distribution as $(n_A(0),r_A(0))$.
\end{proposition}
\begin{proof}
The proof follows the line of the construction of the percolation 
process cluster by cluster. 

\smallskip
\noindent
\textsc{Step~1:} Let $(A^x)_{x\in\Lambda_R}$ be independent realizations of the 
percolation process. To each $x\in\Lambda_R$, we associate its cluster $C_x$ in 
the configuration $A^x$. Let also $\bar C_x$ be the set of all edges of 
$\calE$ with at least one endpoint in common with $C_x$, and set $\partial C_x =
\bar C_x \setminus C_x$.

\smallskip
\noindent
\textsc{Step~2:} Choose an ordering $(x_1, x_2,\ldots,x_{|\Lambda_R|})$ of the 
vertices in $\Lambda_R$ such that $\norminf{x_i}\geq\norminf{x_j}$ whenever 
$i\geq j$. We want to construct a percolation configuration $A$; we thus need 
to decide for each edge $(x,y)$ whether it is open or closed in $A$.

\smallskip
\noindent
\textsc{Step~3:} We start with $x_1=0$. Each edge in $C_{x_1}$ is 
declared open in $A$ and each edge in $\partial C_{x_1}$ is declared to be 
closed in $A$. Set $k=2$ and let $E_{\text{exp}}$ be the set of all edges the 
state of which has already been decided (those of $\bar C_{x_1}$).

\smallskip
\noindent
\textsc{Step~4:}
If $k>|\Lambda_R|$, stop the procedure. Otherwise, let $\bar C_{x_k}'$ be the 
connected component of $x_k$ in $\bar C_{x_k}\setminus E_{\text{exp}}$. Note that 
$\bar C_{x_k}'$ may be empty and that it is not the same as $\bar C_{x_k}'\setminus 
E_{\text{exp}}$; indeed $E_{\text{exp}}$ could separate $\bar C_{x_k}'$ into 
several connected component and we only want to keep the one containing $x_k$.
Declare all edges $e\in\bar C_{x_k}'$ to be open in $A$ if they belong to 
$C_{x_k}$ and closed if they belong to $\partial C_{x_k}$. Let 
$E_{\text{exp}}$ be the set of all edges the state of which  has already been
defined. Increment $k$ and return to \textsc{Step~4}.

\bigskip
Let $A_k$ be the cluster of $x_k$ in $A$ produced using the above procedure, 
and $\partial A_k$ its boundary. Since each edge 
$e\in\bigcup_{k=1}^{|\Lambda_R|} (A_k \cup\partial A_k)$ has been examined 
exactly once, and set to be open or closed independently with probabilities 
$\epsilon J_e$ and $(1-\epsilon) J_e$, respectively, the joint law of all these 
clusters is identical to that under $\bbP$.

We shall need the following three quantities:
\begin{enumerate}
\item[] $N(x)$, the number of points of $C_x$;
\item[] $r(x)=\max\setof{\norminf{y}}{y\lrarrow x}-\norminf{x}$, the outwards
``radius" of $C_x$;
\item[] $b(x_i)=\IF{x_i\not\in\bigcup_{j<i}A_{j}}$;
\item[] $R(x)=\max\setof{\norminf{y-x}}{y\in C_x}$, the ``radius''
of $C_x$. Notice here that $R(x)\geq r(x)$, and that the distribution of this 
quantity is
actually independent of $x$.
\end{enumerate}
By construction, if $x,y$ are in the same cluster of $A$ we have that 
$m_A(x)=m_A(y)$. Moreover, if $b(x)=1$ then $x$ minimizes $\norminf{y}$ among 
all $y\in A_x$; in particular, vertices $x$ with $b(x)=1$ maximize
the ratio $r_A(x)/\norminf{x}$. We thus have
\begin{align*}
\sum_{x\in\couronne} \frac{r_A(x)^2}{\norminf{x}^2}
&\leq
\sum_{x\in\couronne} b(x)n_A(x)\frac{r_A(x)^2}{\norminf{x}^2}\\
&\leq
\sum_{x\in\couronne} b(x)N(x)\frac{r(x)^2}{\norminf{x}^2}
\leq
\sum_{x\in\couronne} N(x)\frac{R(x)^2}{\norminf{x}^2}.
\end{align*}
\qed
\end{proof}

\subsection{Some technical estimates}
\label{App:B}

\begin{lemma}\label{lem_sum(k-l)l}
For all $k\geq 1$ and $\alpha>1$,
\[
\sum_{\ell=1}^{k-1}
\frac1{(k-\ell)^\alpha \ell^\alpha}
\leq
\frac{2^{\alpha+1}\zeta(\alpha)}{k^\alpha}.
\]
\end{lemma}
\begin{proof}
Since $\alpha>1$,
\[
\sum_{\ell=1}^{k-1}\frac1{(k-\ell)^\alpha\ell^\alpha}
\leq
2\sum_{\ell=1}^{(k-1)/2}\frac1{(k-\ell)^\alpha \ell^\alpha}
\leq
2\sum_{\ell=1}^{(k-1)/2}\frac1{(k/2)^\alpha \ell^\alpha}
\leq
2^{1+\alpha} k^{-\alpha} \zeta(\alpha).
\]
\qed
\end{proof}

The next result is classical (see, e.g.~\cite{McDiarmid}); we include its short 
proof for the convenience of the reader.
\begin{lemma}\label{lem_sommeva}
Let $(X_k)_{k\geq 1}$ be independent random variables such that $0\leq 
X_k\leq 1$. Define $S_n=\sum_{k=1}^n X_k$, $\mu\geq\bbE[S_n]$ and  $p=\bbE[S_n]/n$. 
Then, for all $\epsilon>0$,
\[
\bbP(S_n\geq (1+\epsilon)\mu)
\leq 
e^{-((1+\epsilon)\log(1+\epsilon)-\epsilon)\mu} .
\]
\end{lemma}
\begin{proof}
The proof uses a control of the exponential moments of $S_n$ and the Markov 
inequality. Let $h>0$ and recall that for $t\in[0,1]$ we have 
$e^{ht}\leq1-t+te^h$. Using the inequality between arithmetic and geometric 
means,
\begin{align*}
\bbE\bigl[e^{hS_n}\bigr]
=
\prod_{k=1}^n \bbE\bigl[e^{hX_k}\bigr]
&\leq
\prod_{k=1}^n \bigl( 1-\bbE[X_k]+\bbE[X_k]e^h \bigr)\\
&\leq
\Bigl( \frac1n \sum_{k=1}^n \bigl(1-\bbE[X_k] + \bbE[X_k] e^h \bigr)\Bigr)^n\\
&\leq
\Bigl(1-\frac{\bbE[S_n]}{n}+e^h\frac{\bbE[S_n]}{n}\Bigr)^n\\
&\leq
\bigl(1-p+pe^h\bigr)^n .
\end{align*}
Now, we can apply Markov's inequality to obtain, for all $h>0$,
\[
\bbP(S_n\geq m)
\leq
e^{-hm} \bbE\bigl[e^{hS_n}\bigr]
\leq 
e^{-hm} \bigl(1-p+pe^h\bigr)^n .
\]
Setting $m=(1+\epsilon)\mu$ and choosing $h$ such that $e^h=(1+\epsilon)$ yields
\begin{align*}
\bbP(S_n\geq(1+\epsilon)\mu)
&\leq 
(1+\epsilon)^{-(1+\epsilon)\mu}\left(1-p+p(1+\epsilon)\right)^n
\leq
e^{-(1+\epsilon)\log(1+\epsilon)\mu}e^{\epsilon pn}\\
&\leq e^{-(1+\epsilon)\log(1+\epsilon)\mu}e^{\epsilon\mu}.
\end{align*}
\qed
\end{proof}

\nocite{Naddaf}
\bibliographystyle{plain}
\bibliography{bibXY}

\end{document}